\documentclass[10pt,a4paper]{amsart}
\usepackage[utf8]{inputenc}
\usepackage[T1]{fontenc}
\usepackage{amsmath}
\usepackage{upquote}
\usepackage{amsfonts}
\usepackage{amssymb}
\usepackage[english]{babel}
\usepackage{tabularx}
\usepackage{mathtools}
\usepackage{adjustbox}
\usepackage{graphics}
\usepackage{hyperref}
\usepackage{xcolor}
\usepackage{color}
\usepackage{tikz}
\usepackage{chngcntr}
\usetikzlibrary{calc}
\numberwithin{equation}{section}
\newtheorem{thm}{Theorem}[section]
\newtheorem{pr}[thm]{Proposition}
\newtheorem{lm}[thm]{Lemma}
\newtheorem{re}[thm]{Remark}

\newtheorem{ex}[thm]{Example}
\newtheorem{cor}[thm]{Corollary}

\newtheorem{pro}[thm]{Problem}

\newtheoremstyle{case}{}{}{}{}{}{:}{ }{}
\theoremstyle{case}

\counterwithin*{case}{thm}
\newtheoremstyle{caso}{}{}{}{}{}{:}{ }{}
\theoremstyle{caso}

\DeclareRobustCommand{\svdots}{
  \vbox{%
    \baselineskip=0.33333\normalbaselineskip
    \lineskiplimit=0pt
    \hbox{.}\hbox{.}\hbox{.}%
    \kern-0.2\baselineskip
  }%
}

\newcommand{\lcm}{\text{lcm}}

\newcommand{\stirlingone}[2]{\genfrac{[}{]}{0pt}{}{#1}{#2}}
\newcommand{\stirlingtwo}[2]{\genfrac{\lbrace}{\rbrace}{0pt}{}{#1}{#2}}

\theoremstyle{remark}

\makeatletter
\let\@@pmod\pmod
\DeclareRobustCommand{\pmod}{\@ifstar\@pmods\@@pmod}
\def\@pmods#1{\mkern4mu({\operator@font mod}\mkern 6mu#1)}
\makeatother

\title{The Turán and Laguerre inequalities for quasi-polynomial-like functions} 
\author{Krystian Gajdzica}
\address{Institute of Mathematics \\
	Faculty of Mathematics and Computer Science \\
	Jagiellonian University in Cracow
}
\email{krystian.gajdzica@doctoral.uj.edu.pl}



\begin{document}

\setlength{\parindent}{10mm}
\maketitle

\begin{abstract}
This paper deals with both the higher order Tur\'an inequalities and the Laguerre inequalities for quasi-polynomial-like functions --- that are expressions of the form $f(n)=c_l(n)n^l+\cdots+c_d(n)n^d+o(n^d)$, where $d,l\in\mathbb{N}$ and $d\leqslant l$. A natural example of such a function is the $A$-partition function $p_{A}(n)$, which enumerates the number of partitions of $n$ with parts in the fixed finite multiset $A=\{a_1,a_2,\ldots,a_k\}$ of positive integers. For an arbitrary positive integer $d$, we present efficient criteria for both the order $d$ Tur\'an inequality and the $d$th Laguarre inequality for quasi-polynomial-like functions. In particular, we apply these results to deduce non-trivial analogues for $p_A(n)$.

\end{abstract}

\section{Introduction}

A partition of a non-negative integer $n$ is a weakly-decreasing sequence of positive integers $\lambda=(\lambda_1,\lambda_2,\ldots,\lambda_j)$ such that 
$$n=\lambda_1+\lambda_2+\cdots+\lambda_j.$$ 
The numbers $\lambda_i$ are called parts of the partition $\lambda$. The partition function $p(n)$ enumerates all partitions of $n$. For instance, there are $5$ partitions of $4$, namely, $(4)$, $(3,1)$, $(2,2)$, $(2,1,1)$ and $(1,1,1,1)$ --- in other words $p(4)=5$. We do not know any easy formula for $p(n)$. However, Euler proved that its generating function takes the form
\begin{align*}
\sum_{n=0}^\infty p(n)x^n=\prod_{i=1}^\infty\frac{1}{1-x^i}.
\end{align*}

The partition theory plays a crucial rule in many parts of mathematics and other sciences. In statistical mechanics, the well-known Rogers–Ramanujan identities are related to the solution of the hard hexagon model, see \cite{GA3, BA}. Further, partitions have applications in molecular chemistry, crystallography and quantum mechanics, as a consequence of the fact that all irreducible representations of the permutation group $S_n$ and the unitary group $U(n)$ might be labelled by them. It is also worth noting that partitions appear in genetics in the so-called Ewens's sampling formula, see \cite{Ewens, Kingman}. There is a plethora of works devoted to the theory of partitions. For a general introduction to the topic, we encourage the reader to see Andrews' books \cite{GA2, GA1} as well as \cite{AK, H, Sills}.

Now, let us assume that $A=\{a_1,a_2,\ldots,a_k\}$ is a finite multiset of positive integers. By an $A$-partition of a non-negative integer $n$, we mean any partition $\lambda=(\lambda_1,\lambda_2,\ldots,\lambda_j)$ of $n$ with parts in $A$. 
For the sake of clarity, we additionally assume that two $A$-partitions are considered the same if there is only a difference in the order of their parts. The $A$-partition function $p_A(n)$ enumerates all $A$-partitions of $n$. In particular, we have that $p_A(n)=0$ whenever $n$ is a negative integer and $p_A(0)=1$ with $\lambda=()$. The generating function for $p_A(n)$ is given by
\begin{equation}\label{GF}
    \sum_{n=0}^\infty p_A(n)x^n=\prod_{a\in A}\frac{1}{1-x^{a}}.
\end{equation}
For example, if $A=\{1,2,\textcolor{blue}{2},3,\textcolor{blue}{3},\textcolor{red}{3},4,\textcolor{blue}{4}\} $, then we have that $p_A(4)=11$, namely: $(\textcolor{blue}{4})$, $(4)$, $(\textcolor{red}{3},1)$, $(\textcolor{blue}{3},1)$, $(3,1)$, $(\textcolor{blue}{2},\textcolor{blue}{2})$, $(\textcolor{blue}{2},2)$, $(2,2)$, $(\textcolor{blue}{2},1,1)$, $(2,1,1)$ and $(1,1,1,1)$.


There is an abundance of literature devoted to $A$-partition function when $\#A<\infty$. We refer the reader to, for instance, \cite{GA, Bell, CN, DV, KG1, MBN, RS, MU}.

It turns out that $p_A(n)$ is a quasi-polynomial whenever $A$ is a finite set or a multiset of positive integers. More precisely, if $\#A=k$, then the $A$-partition function is an expression of the form 
\begin{align}\label{def: quasi p_A}
    p_A(n)=b_{k-1}(n)n^{k-1}+b_{k-2}(n)n^{k-2}+\cdots+b_0(n),
\end{align}
where the coefficients $b_0(n),b_1(n),\ldots,b_{k-1}(n)$ depend on the residue class of \linebreak $n\pmod*{\lcm{A}}$. The first proof of the above fact is probably due to Bell \cite{Bell}. We encourage the reader to see Stanley's book \cite[Section~4.4]{Stanley} for more information about quasi-polynomials. On the other hand, a quasi-polynomial-like function $f(n)$ is a function which asymptotically behaves like a quasi-polynomial. More specifically, $f(n)$ can be written as
\begin{align}\label{def: quasi-polynomial-like function}
    f(n)=c_{l}(n)n^{l}+c_{l-1}(n)n^{l-1}+\cdots+c_r(n)n^r+o(n^r),
\end{align}
where $r,l\in\mathbb{N}$, $l\geqslant r$, the coefficients $c_r(n),c_{r+1}(n),\ldots,c_{l}(n)$ depend on the residue class of $n\pmod*{M}$ for some positive integer $M\geqslant2$. In particular, we see that $p_A(n)$ is a quasi-polynomial-like function.

This paper deals with two problems. The first of them concerns the so-called higher order Tur\'an inequalities for quasi-polynomial-like functions. Let us recall that a sequence $\left(\omega_i\right)_{i=0}^\infty$ of real numbers satisfies the second order Tur\'an inequality if we have that
$$\omega_n^2\geqslant \omega_{n+1}\omega_{n-1}$$
for all $n\geqslant1$. Further, it fulfills the third order Tur\'an inequality if the following 
\begin{align*}
    4(\omega_n^2-\omega_{n-1}\omega_{n+1})(\omega_{n+1}^2-\omega_n\omega_{n+2})\geqslant(\omega_{n}\omega_{n+1}-\omega_{n-1}\omega_{n+2})^2
\end{align*}
is true for every $n\geqslant1$. More generally, if $J_\omega^{d,n}(x)$ are the Jensen polynomials of degree $d$ and shift $n$ associated to the sequence $\omega:=(\omega_i)_{i=0}^\infty$, defined by
\begin{align*}
    J_\omega^{d,n}(x):=\sum_{i=0}^d\binom{d}{i}\omega_{n+i}x^i,
\end{align*}
then it is known that $(\omega_i)_{i=0}^\infty$ satisfies the order $d$ Tur\'an inequality at $n$ if and only if $J_\omega^{d,n}(x)$ is hyperbolic, i.e. all of its roots are real numbers (see, \cite{Craven, Csordas, Csordas1, Griffin}). 

In 2015 DeSalvo and Pak \cite{DSP} reproved the result (obtained independently by Nicolas \cite{N} in the `70s) that the partition function $p(n)$ satisfies the second order Tur\'an inequality for all $n>25$. Afterwards, Chen \cite{Chen1} conjectured that the third order Tur\'an inequality for $p(n)$ is valid for all $n\geqslant94$. The problem was solved by Chen, Jia and Wang \cite{Chen2} and motivated them to state another conjecture that for each $d\geqslant1$ there is some integer $N_p(d)$ such that the associated Jensen polynomial $J_{p}^{d,n}(X)$ is hyperbolic for all $n\geqslant N_p(d)$. That conjecture, on the other hand, was established by Griffin et al. \cite{Griffin}. It is worth pointing out that Larson and Wagner \cite{Larson} discovered efficient upper bound for the value of $N_p(d)$ for any $d$. 

The aforementioned results have initiated vast research related to discovering similar properties for other variations of the partition function.  Iskander et al. \cite{Iskander} proved that for every $d\geqslant2$ the fractional partition function $p_\alpha(n)$, which is defined for $\alpha\in\mathbb{Q}$ in terms of the following generating function
\begin{align*}
    \sum_{n=0}^\infty p_\alpha(n)x^n:=\prod_{i=1}^\infty\frac{1}{(1-x^i)^\alpha}
\end{align*}
(for more information, see \cite{Chan}), satisfies the order $d$ Tur\'an inequality for all but finitely many values of $n$. Further, Craig and Pun \cite{Craig} investigated the so-called $k$-regular partition function $p_{k}(n)$ (i.e. $p_{k}(n)$ enumerates only those partitions of $n$ whose parts are not divisible by $k$) in that context. They obtained that for every $k\geqslant2$ and $d\geqslant1$ the associated Jensen polynomial $J_{p_{k}}^{d,n}(X)$ is hyperbolic for all sufficiently large numbers $n$. Heim, Neuhauser and Tr\"{o}ger \cite{BNT2} investigated the plane partition function $PL(n)$ (see Andrews \cite[Chapter~11]{GA2} or \cite[Chapter~10]{GA1}) and its polynomization in this direction. They conjectured that for any $d\geqslant1$ the plane partition function fulfills the order $d$ Tur\'an inequality for all large enough numbers $n$. That conjecture was solved by Ono, Pujahari and Rolen in \cite{Ono} with explicit bounds provided by Ono’s PhD student Pandey \cite{Pandey}. Further, Baker and Males \cite{Baker} showed that the number $\overline{p}_j(n)$ of partitions with BG-rank $j$, and the number $\overline{p}_j(a,b;n)$ of partitions with BG-rank $j$ and $2$-quotient rank congruent to $a\pmod*{b}$ satisfy (asymptotically) all higher order Tur\'an inequalities for even values of $j$ and $n$. We refer the reader to Berkovich and Garvan's paper \cite{Berkovich} for additional information about $\overline{p}_j(n)$ and $\overline{p}_j(a,b;n)$. Finally, Dong, Ji and Jia \cite{Dong} discovered that the Jensen polynomial corresponding to $d\geqslant1$ and the Andrews and Paule’s broken $k$-diamond partition function $\Delta_k(n)$, namely $J_{\Delta_k}^{d,n}(X)$, is hyperbolic for $k=1$ or $2$ and all but finitely many positive integers $n$. The explicit definition of broken $k$-diamond partitions (for any $k\geqslant1$) together with some properties of $\Delta_k(n)$ might be found in Andrews and Paule's paper \cite{GA4}. The above-mentioned results have been our motivation to study the higher order Tur\'an inequalities for both quasi-polynomial-like functions in general and $A$-partition functions in particular.

The second issue which this paper deals with concerns the so-called Laguerre inequalities for quasi-polynomial-like functions. Once again, let us assume that $\omega=\left(\omega_i\right)_{i=0}^\infty$ is a sequence of real numbers. For a fixed non-negative integer $d$, we say that $\omega$ satisfies the Laguerre inequality of order $d$ at $n$ if
\begin{align}\label{def: Discrete Laguerre}
    \sum_{j=0}^{2d}(-1)^{j+d}\binom{2d}{j}\omega_{n+j}\omega_{n+2d-j}\geqslant0.
\end{align}
The discrete Laguerre inequalities (\ref{def: Discrete Laguerre}) were firstly introduced by Wang and Yang \cite{Wang-Yang}. It is also worth noting that Wagner \cite[Theorem 1.4]{Wagner2} defined them equivalently by dividing (\ref{def: Discrete Laguerre}) by $(2d)!$. For $d=1$, one can easy observe that (\ref{def: Discrete Laguerre}) reduces to the second order Tur\'an inequality. If $d=2$, then (after simplification) we get 
\begin{align*}
    3 \omega_{n+2}^2-4 \omega_{n+1} \omega_{n+3}+\omega_n \omega_{n+4}\geqslant0.
\end{align*}
Further, the order $3$ Laguerre inequality might be written equivalently as follows:
\begin{align*}
    10 \omega_{n+3}^2-15 \omega_{n+2} \omega_{n+4}+6 \omega_{n+1} \omega_{n+5}-\omega_n \omega_{n+6}\geqslant0,
\end{align*}
and so on.

Wang and Yang \cite{Wang-Yang, Wang-Yang2} investigated Lagurre inequalities for many combinatorial sequences. In particular, they showed that the partition function, the overpartition function, the Motzkin numbers, the Fine numbers, the Domb numbers and the distinct partition function satisfy the order $2$ Laguerre inequality. More recently, Yang \cite{Yang3} also proved that the broken $k$-diamond partition function fulfills the second order Laguerre inequality. On the other hand, Wagner \cite{Wagner2} showed that the partition function satisfies the inequality (\ref{def: Discrete Laguerre}) for every non-negative integer $d$ and all sufficiently large values of $n$. The aforementioned results have motivated us to investigate the issue in the case of quasi-polynomial-like functions.

At the end of Introduction, it needs to be pointed out that studying both the higher order Tur\'an inequalities and the Laguerre inequalities is not only art for art's sake. Let us recall that a real entire (i.e. analytic at all points of the complex plane $\mathbb{C}$) function
\begin{align*}
    f(x)=\sum_{n=0}^\infty a_n\frac{x^n}{n!}
\end{align*}
is in the $\mathcal{LP}$ (Laguerre-P\'olya) class if it may be written as
\begin{align*}
    f(x)=cx^ke^{-ax^2+bx}\prod_{n=1}^\infty\left(1+\frac{x}{x_n}\right)e^{-\frac{x}{x_n}},
\end{align*}
where $a,b,c,x_1,x_2,\ldots$ are all real numbers with $a\geqslant0$, $k$ is a non-negative integer and $\sum_{n=1}^\infty x_n^{-2}<\infty$. For the background of the theory of the $\mathcal{LP}$ functions, we encourage the reader to see \cite{Levin, Rahman}. It turns out that the Riemann hypothesis is equivalent to the statement that the Riemann $\Xi$-function 
\begin{align*}
    \Xi(z):=\frac{1}{2}\left(-z^2-\frac{1}{4}\right)\pi^{\frac{iz}{2}-\frac{1}{4}}\Gamma\left(-\frac{iz}{2}+\frac{1}{4}\right)\zeta\left(-iz+\frac{1}{2}\right)
\end{align*}
is in the $\mathcal{LP}$ class, where $\Gamma$ is the gamma function, and $\zeta$ denotes the Riemann zeta function. There is a necessary condition for the Riemann $\Xi$-function to be in the Laguerre–P\'olya class which states that the Maclaurin coefficients of the $\Xi$-function have to fulfill the order $d$ Tur\'an inequality as well as the Laguerre inequality of order $d$ for every positive integer $d$. For additional information, we refer the reader to \cite{Dimitrov, Patrick, Szego}.

This manuscript is organized as follows. Section 2 delivers necessary concepts, notations and properties which are used throughout the paper. Section 3 studies the higher order Tur\'an inequalities for both quasi-polynomial-like functions and $A$-partition functions. In Section 4, on the other hand, we deal with the Laguerre inequalities. Finally, Section 5 contains some concluding remarks and open problems.

\section{Preliminaries}

At first, we fix some notation. The set of non-negative integers is denoted by $\mathbb{N}$. Further, we put $\mathbb{N}_+:=\mathbb{N}\setminus\{0\}$ and $\mathbb{N}_{\geqslant k}:=\mathbb{N}\setminus\{0,1,\ldots,k-1\}$. 

For a finite multiset $A=\{a_1,a_2,\ldots,a_k\}$ of positive integers, we associate the $A$-partition function $p_A(n)$, which was defined in Introduction. Due to Bell's theorem \cite{Bell}, we know that $p_A(n)$ is a quasi-polynomial given by the equality $(\ref{def: quasi p_A})$, where the coefficients $b_0(n),b_1(n),\ldots,b_{k-1}(n)$ depend on the residue class of $n\pmod*{\lcm{A}}$. It turns out that under some additional assumptions on $A$, we may determine some of the coefficients $b_i(n)$. That is a result obtained by several authors, among others, Almkvist \cite{GA}, Beck et al. \cite{Beck} or Israilov \cite{Israilov}. We present the theorem due to Almkvist \cite{GA}. In order to do that, let us define symmetric polynomials $\sigma_i(x_1,x_2,\ldots,x_k)$ in terms of the power series expansion
\begin{align*}
\sum_{m=0}^\infty\sigma_m(x_1,x_2,\ldots,x_k)t^m:=\prod_{i=1}^k\frac{x_it/2}{\sinh(x_it/2)}.
\end{align*} 
Now, we have the following.
\begin{thm}[Almkvist]\label{2.4}
Let $A=\{a_1,a_2,\ldots,a_k\}$ be fixed and put $s_1:=a_1+a_2+\cdots+a_k$. For a given integer $1\leqslant j\leqslant k$, if $\gcd B=1$ for every $j$-element multisubset ($j$-multisubset) $B$ of $A$, then
\begin{equation*}
p_A(n)=\frac{1}{\prod_{i=1}^ka_i}\sum_{i=0}^{k-j}\sigma_i(a_1,a_2,\ldots,a_k)\frac{(n+s_1/2)^{k-1-i}}{(k-1-i)!}+O(n^{j-2}).
\end{equation*}
\end{thm}
One can check that $\sigma_i=0$ if $i$ is odd. Furthermore, if we set $s_m:=a_1^m+a_2^m+\cdots+a_k^m$, then
\begin{align*}
\sigma_0=1\text{,}\hspace{0.2cm}\sigma_2=-\frac{s_2}{24}\text{,}\hspace{0.2cm}\sigma_4=\frac{5s_2^2+2s_4}{5760}\text{,}\hspace{0.2cm}\sigma_6=-\frac{35s_2^3+42s_2s_4+16s_6}{2903040}.
\end{align*}
Essentially, Theorem \ref{2.4} maintains that if $\gcd B=1$ for every $(k-j)$-multisubset $B$ of $A$, then the coefficients $b_{k-1}(n),b_{k-2}(n),\ldots,b_{k-1-j}(n)$ in the equality (\ref{def: quasi p_A}) are independent of the residue class of $n\pmod*{\lcm{A}}$, i.e. they are constants and can be explicitly calculated. Moreover, it is noteworthy that the $A$-partition function is a non-trivial example of a quasi-polynomial-like function --- that is an expression of the form $(\ref{def: quasi-polynomial-like function})$.

Now, let us recall some terminology related to higher order Tur\'an inequalities. Instead of repeating the discussion from Introduction, we directly explain how the order $d$ Tur\'an inequality arises from the hyperbolicity of the Jensen polynomial $J_\omega^{d,n}(x)$ has to be hyperbolic. Let
\begin{align*}
    g(x)=c_sx^s+c_{s-1}x^{s-1}+c_{s-2}x^{s-2}+\cdots+c_{0}
\end{align*}
be a fixed polynomial with real coefficients and denote all its complex roots by $\alpha_1,\alpha_2,\ldots,\alpha_{s}$. By $P_m$, we mean the $m$-th Newton's sum of $g(x)$, which is given by
\begin{align*}
    P_m=\begin{cases}
  s,  & \text{if } m=0, \\
  \alpha_1^m+\alpha_2^m+\cdots+\alpha_s^m,  & \text{if } m=1,2,3,4,\ldots.
\end{cases}
\end{align*}
Further, for the sums $P_0,\ldots,P_{2s-2}$, we associate the Hankel matrix $H(g)$, namely
\begin{align*}
    H(g):=\begin{bmatrix}
P_0 & P_1 & P_2 & \cdots & P_{s-1}\\
P_1 & P_2 & P_3 & \cdots & P_{s}\\
P_2 & P_3 & P_4 & \cdots & P_{s+1}\\
\vdots & \vdots & \vdots & \vdots & \vdots\\
P_{s-2} & P_{s-1} & P_{s} &  \cdots & P_{2s-3}\\
P_{s-1} & P_{s} & P_{s+1} & \cdots & P_{2s-2}
\end{bmatrix}.
\end{align*}
The classical Hermit's theorem \cite{Hermit} states that $g(x)$ is hyperbolic if and only if the matrix $H(g)$ is positive semi-definite. Since each of the Newton's sums might be expressed in terms of the coefficients $c_s,c_{s-1},\ldots,c_0$, Hermit's result provides a set of inequalities on them by

\begin{align*}
    \det\begin{bmatrix}
        P_0
    \end{bmatrix}\geqslant0,\det\begin{bmatrix}
        P_0 & P_1\\
        P_1 & P_2
    \end{bmatrix}\geqslant0,\ldots,\det\begin{bmatrix}
       P_0 & P_1 & P_2 & \cdots & P_{s-1}\\
P_1 & P_2 & P_3 & \cdots & P_{s}\\
P_2 & P_3 & P_4 & \cdots & P_{s+1}\\
\vdots & \vdots & \vdots & \vdots & \vdots\\
P_{s-2} & P_{s-1} & P_{s} &  \cdots & P_{2s-3}\\
P_{s-1} & P_{s} & P_{s+1} & \cdots & P_{2s-2}
    \end{bmatrix}\geqslant0.
\end{align*}

Now, if we assign the Jensen polynomial $J_\omega^{d,n}(x)$ for an arbitrary sequence $\omega=(w_i)_{i=1}^\infty$, then the corresponding inequality for the determinant of the main minor $l\times l$ of $H(J_\omega^{d,n})$:
\begin{align*}
    \det\begin{bmatrix}
       P_0 & P_1 & P_2 & \cdots & P_{l-1}\\
P_1 & P_2 & P_3 & \cdots & P_{l}\\
P_2 & P_3 & P_4 & \cdots & P_{l+1}\\
\vdots & \vdots & \vdots & \vdots & \vdots\\
P_{l-2} & P_{l-1} & P_{l} &  \cdots & P_{2l-3}\\
P_{l-1} & P_{l} & P_{l+1} & \cdots & P_{2l-2}
    \end{bmatrix}\geqslant0
\end{align*}
is called the order $l$ Tur\'an inequality for the sequence $\omega$. In particular, it means that $J_\omega^{d,n}(x)$ is hyperbolic if and only if the sequence $\omega_n=(w_{n+j})_{j=1}^\infty$ satisfies the order $l$ Tur\'an inequality for every $l\in\{1,2,\ldots,d\}$.

From the above discussion, we see that investigating the higher order Tur\'an inequalities does not seem to be an easy challenge. However, there is a paper due to Griffin, Ono, Rolen and Zagier \cite{Griffin}, which delivers an efficient criterion to deal with that issue.

\begin{thm}[Griffin, Ono, Rolen, Zagier]\label{theorem: Griffin-Ono-Rolen-Zagier}
    Let $(\omega_n)_{n=0}^\infty$ be a sequence of real numbers. Suppose further that $(E(n))_{n=0}^\infty$ and $(\delta(n))_{n=0}^\infty$ are sequences of positive real numbers with $\lim_{n\to\infty}\delta(n)=0$, and that $F(t)=\sum_{i=0}^\infty c_it^i$ is a formal power series with complex coefficients. For a fixed $d\geqslant1$, suppose that there are sequences $\left(C_0(n)\right)_{n=0}^\infty,\left(C_1(n)\right)_{n=0}^\infty,\ldots,\left(C_d(n)\right)_{n=0}^\infty$ of real numbers, with $\lim_{n\to\infty} C_i(n)=c_i$ for $0\leqslant i \leqslant d$, such that for $0\leqslant j \leqslant d$, we have
    \begin{align*}
        \frac{\omega_{n+j}}{\omega_n}E(n)^{-j}=\sum_{i=0}^dC_i(n)\delta(n)^ij^i+o\left(\delta(n)^d\right)\hspace{1cm}\text{as } n\to\infty.
    \end{align*}
    Then, we have
    \begin{align*}
        \lim_{n\to\infty}\left(\frac{\delta(n)^{-d}}{\omega_n}J_\omega^{d,n}\left(\frac{\delta(n)x-1}{E(n)}\right)\right)=H_{F,d}(x),
    \end{align*}
    uniformly for $x$ in any compact subset of $\mathbb{R}$, where the polynomials $H_{F,m}(x)\in\mathbb{C}[x]$ are defined either by the generating function $F(-t)e^{xt}=\sum_{m=0}^\infty H_{F,m}(x)t^m/m!$ or in closed form by $H_{F,m}(x):=m!\sum_{l=0}^m (-1)^{m-l}c_{m-l}x^l/l!$.
\end{thm}

It is not clear how one can apply the above result in practice. In fact, Griffin et al. use the criterion to prove that for every positive integer $d$ the partition function $p(n)$ fulfills the order $d$ Tur\'an inequality for all but finitely many values of $n$. More precisely, they obtain the Hermite polynomials $H_m(x)$ as the polynomials $H_{F,m}(x)$ in Theorem \ref{theorem: Griffin-Ono-Rolen-Zagier}. Let us recall that they define the Hermit polynomials via the generating function
\begin{align*}
    \sum_{j=0}^\infty H_j(x)\frac{t^j}{j!}=e^{-j^2+jx}=1+jx+\frac{j^2}{2!}(x^2-2)+\cdots.
\end{align*}
Since these polynomials have only distinct real roots, and since the property of a polynomial with only real roots is invariant under small deformation, the required phenomenon for $p(n)$ follows.

On the other hand, we investigate the higher order Tur\'an inequalities for quasi-polynomial-like functions 
\begin{align*}
    f(n)=c_{l}(n)n^{l}+c_{l-1}(n)n^{l-1}+\cdots+c_r(n)n^r+o(n^r),
\end{align*}
where $r,l\in\mathbb{N}$, $l\geqslant r$ and the coefficients $c_r(n),c_{r+1}(n),\ldots,c_{l}(n)$ depend on the residue class of $n\pmod*{M}$ for some positive integer $M\geqslant2$. Therefore, we will probably get another family of orthogonal polynomials in Theorem \ref{theorem: Griffin-Ono-Rolen-Zagier}. The generalized Laguerre polynomials $L_n^{(\alpha)}(x)$ for $\alpha>-1$ are defined via the following conditions of orthogonality and normalization
\begin{align*}
    \int_0^\infty e^{-x}x^\alpha L_n^{(\alpha)}(x)L_m^{(\alpha)}(x)dx=\Gamma(\alpha+1)\binom{n+\alpha}{n}\delta_{n,m},
\end{align*}
where $\Gamma$ denotes the Euler gamma function, $\delta_{i,j}$ is the Kronecker delta and $n,m=0,1,2,\ldots.$ Moreover, we demand that the coefficient of $x^n$ in the polynomial $L_n^{(\alpha)}(x)$ of degree $n$ have the sign $(-1)^n$. One can figure out the explicit representation of these polynomials, namely,
\begin{align*}
    L_n^{(\alpha)}(x)=\sum_{j=0}^n\binom{n+\alpha}{n-j}\frac{(-x)^j}{j!}.
\end{align*}
Hence, we have that
\begin{align*}
    L_0^{(\alpha)}(x)&=1,\\
    L_1^{(\alpha)}(x)&=-x+(\alpha+1),\\
    L_2^{(\alpha)}(x)&=\frac{x^2}{2}-(\alpha+2)x+\frac{(\alpha+1)(\alpha+2)}{2},\\
    L_3^{(\alpha)}(x)&=\frac{-x^3}{6}+\frac{(\alpha+3)x^2}{2}-\frac{(\alpha+2)(\alpha+3)x}{2}+\frac{(\alpha+1)(\alpha+2)(\alpha+3)}{6},
\end{align*}
and so on. It is well-known that if $\alpha$ is non-negative, then $L_n^{(\alpha)}(x)$ has exactly $n$ positive real roots. For more information about both the Hermite polynomials and the Laguerre polynomials we encourage the reader to see \cite{Szego2}. 

Finally, instead of repeating the text from Introduction related to the Laguerre inequalities, we just recall that for an arbitrary sequence $\omega=\left(\omega_i\right)_{i=0}^\infty$ of real numbers the Laguerre inequality of order $d$ at $n$ is defined via
\begin{align*}
    \sum_{j=0}^{2d}(-1)^{j+d}\binom{2d}{j}\omega_{n+j}\omega_{n+2d-j}\geqslant0.
\end{align*}
In order to deal with this issue for quasi-polynomial-like functions we will need some basic identities involving binomial coefficients, which are omitted here and collected in Section 4.

Now, we are ready to proceed to the main part of the manuscript.

\section{The higher order Tur\'an inequalities for quasi-polynomial-like functions}

The main goal of this section is to prove the following characterization.

\begin{thm}\label{theorem: Criterion for quasi-polynomial-like functions}
    Let $f(n)$ be a quasi-polynomial-like function of the form
    \begin{align*}
        f(n)=c_{l}n^{l}+c_{l-1}n^{l-1}+\cdots+c_{l-d}n^{l-d}+o(n^{l-d}),
    \end{align*}
    for some $1\leqslant d\leqslant l$. Then, for every $1 \leqslant j \leqslant d$ the sequence $(f(n))_{n=0}^\infty$ satisfies the order $j$ Tur\'an inequality for all but finitely many values of $n$.
\end{thm}
\begin{proof}
    At first, let us fix $0\leqslant j\leqslant d$ and expand $f(n+j)/f(n)$. We have that
    \begin{align*}
        \frac{f(n+j)}{f(n)}&=\frac{c_{l}(n+j)^{l}+c_{l-1}(n+j)^{l-1}+\cdots+c_{l-d}(n+j)^{l-d}+o((n+j)^{l-d})}{c_{l}n^{l}+c_{l-1}n^{l-1}+\cdots+c_{l-d}n^{l-d}+o(n^{l-d})}\\
        &=\frac{c_{l}n^{l}+c_{l-1}n^{l-1}\cdots+c_{l-d}n^{l-d}+o(n^{l-d})}{c_{l}n^{l}+\cdots+c_{l-d}n^{l-d}+o(n^{l-d})}\\
        &+j\cdot\frac{lc_{l}n^{l-1}+(l-1)c_{l-1}n^{l-2}+\cdots+(l-d)c_{l-d}n^{l-d-1}+o(n^{l-d})}{c_{l}n^{l}+\cdots+c_{l-d}n^{l-d}+o(n^{l-d})}\\
        &+j^2\cdot\frac{\binom{l}{2}c_{l}n^{l-2}+\binom{l-1}{2}c_{l-1}n^{l-3}+\cdots+\binom{l-d}{2}c_{l-d}n^{l-d-2}+o(n^{l-d})}{c_{l}n^{l}+\cdots+c_{l-d}n^{l-d}+o(n^{l-d})}\\
        &\vdots\\
        &+j^d\cdot\frac{\binom{l}{d}c_{l}n^{l-d}+\binom{l-1}{d}c_{l-1}n^{l-1-d}+\cdots+\binom{l-d}{d}c_{l-d}n^{l-2d}+o(n^{l-d})}{c_{l}n^{l}+\cdots+c_{l-d}n^{l-d}+o(n^{l-d})}\\
        &+\frac{o((n+j)^{l-d})}{c_{l}n^{l}+\cdots+c_{l-d}n^{l-d}+o(n^{l-d})}\\
        &=\frac{c_{l}n^{l}+c_{l-1}n^{l-1}\cdots+c_{l-d}n^{l-d}+o(n^{l-d})}{c_{l}n^{l}+\cdots+c_{l-d}n^{l-d}+o(n^{l-d})}\\
        &+j\cdot\frac{1}{n}\frac{lc_{l}n^{l-1}+(l-1)c_{l-1}n^{l-2}+\cdots+(l-d)c_{l-d}n^{l-d-1}+o(n^{l-d})}{c_{l}n^{l-1}+\cdots+c_{l-d-1}n^{l-d-1}+o(n^{l-d-1})}\\
        &+j^2\cdot\frac{1}{n^2}\frac{\binom{l}{2}c_{l}n^{l-2}+\binom{l-1}{2}c_{l-1}n^{l-3}+\cdots+\binom{l-d}{2}c_{l-d}n^{l-d-2}+o(n^{l-d})}{c_{l}n^{l-2}+\cdots+c_{l-d}n^{l-d-2}+o(n^{l-d-2})}\\
        &\vdots\\
        &+j^d\cdot\frac{1}{n^d}\frac{\binom{l}{d}c_{l}n^{l-d}+\binom{l-1}{d}c_{l-1}n^{l-1-d}+\cdots+\binom{l-d}{d}c_{l-d}n^{l-2d}+o(n^{l-d})}{c_{l}n^{l-d}+\cdots+c_{l-2d}n^{l-2d}+o(n^{l-2d})}\\
        &+\frac{o((n+j)^{l-d})}{c_{l}n^{l}+\cdots+c_{l-d}n^{l-d}+o(n^{l-d})}.
    \end{align*}
    Now, it is not difficult to see that we can apply Theorem \ref{theorem: Griffin-Ono-Rolen-Zagier} with $\omega_n=f(n)$, $E(n)=1$, and $\delta(n)=n^{-1}$. Indeed, we get
    \begin{align*}
        \frac{f(n+j)}{f(n)}=\sum_{i=0}^dC_i(n)\left(\frac{1}{n}\right)^ij^i+o\left(\left(\frac{1}{n}\right)^d\right)\hspace{1cm}\text{as } n\to\infty,
    \end{align*}
    where
    \begin{align*}
        C_s(n)=\frac{\binom{l}{s}c_{l}n^{l-s}+\binom{l-1}{s}c_{l-1}n^{l-1-s}+\cdots+\binom{l-d}{s}c_{l-d}n^{l-d-s}+o(n^{l-d})}{c_{l}n^{l-s}+\cdots+c_{l-d}n^{l-d-s}+o(n^{l-d-s})}
    \end{align*}
    for any $0\leqslant s\leqslant d$. Hence, it is clear that 
    \begin{align*}
        \lim_{n\to\infty}C_s(n)=\binom{l}{s}
    \end{align*}
    for every $0\leqslant s\leqslant d$, and we obtain that
    \begin{align*}
        H_{F,m}(x)=m!\sum_{j=0}^m(-1)^{m-j}\binom{l}{m-j}\frac{x^j}{j!}&=(-1)^mm!\sum_{j=0}^m\binom{m+(l-m)}{m-j}\frac{(-x)^j}{j!}\\
        &=(-1)^mm!L_m^{(l-m)}(x)
    \end{align*}
    for each $0\leqslant m \leqslant d$, where $L_m^{(l-m)}(x)$ is the generalized Laguerre polynomial. Since $l-m\geqslant0$, the polynomials $L_m^{(l-m)}(x)$ have only positive real roots (see, the antepenultimate paragraph of Section 2). Finally, Theorem \ref{theorem: Griffin-Ono-Rolen-Zagier} asserts that 
    \begin{align*}
        \lim_{n\to\infty}\left(\frac{n^s}{f(n)}J_f^{s,n}\left(\frac{x}{n}-1\right)\right)=(-1)^ss!L_s^{(l-s)}(x),
    \end{align*}
    uniformly for $x$ in any compact subset of $\mathbb{R}$ for every $1\leqslant s\leqslant d$. However, we know that the property of a polynomial with real coefficients is invariant under small deformation. Thus, the required phenomenon for $f(n)$ follows.
\end{proof}

Theorem \ref{2.4} and Theorem \ref{theorem: Criterion for quasi-polynomial-like functions} deliver an interesting criterion for the order $d$ Tur\'an inequality for the $A$-partition function.
\begin{thm}\label{theorem: higher order Turan A-partition function}
    Let $A$ be a finite multiset (or set) of positive integers with $\#A=k$, and let $1\leqslant d<k$ be fixed. Suppose further that $\gcd B=1$ for every $(k-d)$-multisubset $B\subset A$. Then, for any $1\leqslant j \leqslant d$ the sequence $(p_A(n))_{n=0}^\infty$ fulfills the order $j$ Tur\'an inequality for all sufficiently large values of $n$.
\end{thm}
\begin{proof}
    That is a direct consequence of both Theorem \ref{2.4} and Theorem \ref{theorem: Criterion for quasi-polynomial-like functions}.
\end{proof}

An interesting question arises whether Theorem \ref{theorem: Criterion for quasi-polynomial-like functions} and Theorem \ref{theorem: higher order Turan A-partition function} present also the necessary conditions for order $d$ Tur\'an inequality for both quasi-polynomial-like functions and $A$-partition functions, respectively. It is true for the order $2$ Tur\'an inequality which follows directly from Gajdzica's papers \cite{KG2,KG3}. However, it is not true in general as the forthcoming examples show. 
\begin{ex}\label{example: quasi-polynomial-like counterexample}
    Let us investigate the order $3$ Tur\'an inequality for the function
    \begin{align*}
        f(n)=\begin{cases}
        n^{15}+n^{14}+n^{13}+n^{12}+n^{11}+o(n^{11}), & \text{if } n\not\equiv0 \pmod*{4},\\
         n^{15}+n^{14}+n^{13}+2n^{12}+n^{11}+o(n^{11}), & \text{if } n\equiv0 \pmod*{4}.
        \end{cases}
    \end{align*}
    It is easy to see that the assumptions from Theorem \ref{theorem: Criterion for quasi-polynomial-like functions} are not satisfied. Nevertheless, it turns out that the function which directly corresponds to the third order Tur\'an inequality takes the form
    \begin{align*}
    4&\left(f(n)^2-f(n-1)f(n+1)\right)\left(f(n+1)^2-f(n)f(n+2)\right)\\
    -&\left(f(n)f(n+1)-f(n-1)f(n+2)\right)^2=\begin{cases}
        12411n^{54}+o(n^{54}), & \text{if } n\equiv0 \pmod*{4},\\
        12771n^{54}+o(n^{54}), & \text{if } n\equiv1 \pmod*{4},\\
        12539n^{54}+o(n^{54}), & \text{if } n\equiv2 \pmod*{4},\\
         12659n^{54}+o(n^{54}), & \text{if } n\equiv3 \pmod*{4};
        \end{cases}
\end{align*}
and is positive for all sufficiently large values of $n$. Hence, we conclude that Theorem \ref{def: quasi-polynomial-like function} is not an optimal criterion.
\end{ex}

In the case of the $A$-partition function, we present the following counterexample.

\begin{ex}\label{Example: 1,1,1,1,300}
    Let us assume that $A=\{1,\text{$\textcolor{blue}{1},\textcolor{red}{1},\textcolor{brown}{1}$},300\}$, and examine the order $4$ Tur\'an inequality.  In the fashion of Example \ref{example: quasi-polynomial-like counterexample}, we wish to define a function $f(n)$ which directly corresponds to that issue. It is tedious but elementary to show that a sequence $\omega=\left(\omega_i\right)_{i=0}^\infty$ fulfills the fourth Tur\'an inequality if the following
    \begin{align*}
        &54 \omega_n\omega_{n+1}^2 \omega_{n+2} \omega_{n+4}^2+\omega_n^3 \omega_{n+4}^3+108\omega_{n+1}^3 \omega_{n+2} \omega_{n+3} \omega_{n+4}+36\omega_{n+1}^2 \omega_{n+2}^2 \omega_{n+3}^2\\
        -&12 \omega_n^2 \omega_{n+1}\omega_{n+3} \omega_{n+4}^2-54\omega_{n+1}^2 \omega_{n+2}^3 \omega_{n+4} -6 \omega_n\omega_{n+1}^2 \omega_{n+3}^2 \omega_{n+4}-54 \omega_n \omega_{n+2}^3 \omega_{n+3}^2\\
        -&27\omega_{n+1}^4 \omega_{n+4}^2-180 \omega_n\omega_{n+1} \omega_{n+2}^2 \omega_{n+3} \omega_{n+4}-27 \omega_n^2 \omega_{n+3}^4+108 \omega_n\omega_{n+1} \omega_{n+2} \omega_{n+3}^3\\
        -&64\omega_{n+1}^3 \omega_{n+3}^3-18 \omega_n^2 \omega_{n+2}^2 \omega_{n+4}^2+81 \omega_n \omega_{n+2}^4 \omega_{n+4}+54 \omega_n^2 \omega_{n+2} \omega_{n+3}^2 \omega_{n+4}\geqslant0
    \end{align*}
    is satisfied for every $n\in\mathbb{N}$. Therefore, let us put $\omega_n:=p_A(n)$ and denote the left hand side of the above inequality by $f_A(n)$. Since $\gcd(300)\not=1$, we see that the assumptions from Theorem \ref{theorem: higher order Turan A-partition function} do not hold if $d=4$. Notwithstanding, one can carry out the appropriate computations in Mathematica \cite{WM} and check that $f_A(n)$ is a quasi-polynomial of degree $12$ with coefficients depending on $n\pmod*{300}$. It might be also verified that the leading coefficient of $f_A(n)$ attains the smallest value whenever $n\not\equiv296 \pmod*{300}$ --- in all of these cases, we have
    \begin{align*}
        f_A(n)=\frac{n^{12}}{2^{18}\cdot3^9\cdot5^{12}}+o\left(n^{12}\right).
    \end{align*}
    The above discussion agrees with the plot of $f_A(n)$ for $1\leqslant n\leqslant10^4$, see Figure 1.
    \begin{center}
\begin{figure}[!htb]
   \begin{minipage}{1\textwidth}
     \centering
     \includegraphics[width=1\linewidth]{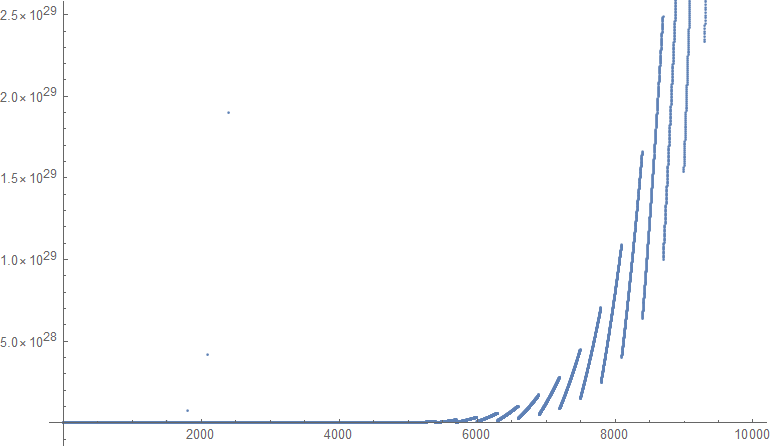}
     \caption{The values of $f_A(n)$ for $A=\{1,\text{$\textcolor{blue}{1},\textcolor{red}{1},\textcolor{brown}{1}$},300\}$ and $1\leqslant n \leqslant10^4$.}
   \end{minipage}\hfill
\end{figure}
\end{center}

Hence, we conclude that Theorem \ref{theorem: higher order Turan A-partition function} is not optimal, as well.
\end{ex}

At the end of this section, let us exhibit two other examples. Sometimes we can not conclude the appropriate order $d$ Tur\'an inequality if the requirements from Theorem \ref{theorem: Criterion for quasi-polynomial-like functions} or Theorem \ref{theorem: higher order Turan A-partition function} are not satisfied.

\begin{ex}
    Let us consider a quasi-polynomial-like function of the form
    \begin{align*}
        f(n)=\begin{cases}
        n^{15}+n^{14}+n^{13}+n^{12}+n^{11}+o(n^{11}), & \text{if } n\not\equiv0 \pmod*{4},\\
         n^{15}+n^{14}+n^{13}+500n^{12}+n^{11}+o(n^{11}), & \text{if } n\equiv0 \pmod*{4}.
        \end{cases}
    \end{align*}
    We would like to investigate the order $3$ Tur\'an inequality. However, it is clear that the assumptions from Theorem \ref{theorem: Criterion for quasi-polynomial-like functions} are not satisfied; and one may calculate that
    \begin{align*}
    4&\left(f(n)^2-f(n-1)f(n+1)\right)\left(f(n+1)^2-f(n)f(n+2)\right)\\
    -&\left(f(n)f(n+1)-f(n-1)f(n+2)\right)^2=-266341n^{54}+o(n^{54}),
\end{align*}
whenever $n\equiv2\pmod*{4}$. Hence, $f(n)$ can not satisfy the order $3$ Tur\'an inequality for all sufficiently large values of $n$, as required.
\end{ex}

As an instance for an $A$-partition function, we take a finite analogue of the partition function $p(n)$.
\begin{ex}\label{Example: 1,2,3,...}
    For any positive integer $m$, let us put $A_m:=\{1,2,\ldots,m\}$. We want to consider the third order Tur\'an inequality for $p_{A_6}(n)$ and $p_{A_7}(n)$. In order to make the text more transparent, we set
    \begin{align*}
        g_{A_m}(n):=4&\left(p_{A_m}^2(n)-p_{A_m}(n-1)p_{A_m}(n+1)\right)\left(p_{A_m}^2(n+1)-p_{A_m}(n)p_{A_m}(n+2)\right)\\
    -&\left(p_{A_m}(n)p_{A_m}(n+1)-p_{A_m}(n-1)p_{A_m}(n+2)\right)^2
    \end{align*}
    Thus, $g_{A_m}(n)$ directly corresponds to the order $3$ Tur\'an inequality. It is clear that the demands from Theorem \ref{theorem: higher order Turan A-partition function} are not true for $A_6$. In fact, it turns out that, for instance,
    \begin{align*}
        g_{A_6}(n)=-\frac{2069n^{14}}{2^{24}\cdot3^{12}\cdot5^6}+o\left(n^{14}\right),
    \end{align*}
    whenever $n\equiv2\pmod*{60}$. On the other hand, one can check that the equality
    \begin{align*}
        g_{A_7}(n)=\frac{n^{18}}{2^{28}\cdot3^{14}\cdot5^7\cdot7^4}+o\left(n^{18}\right)
    \end{align*}
    is valid for every positive integer $n$, as required. The above discussion agrees with the plots of $g_{A_6}(n)$ and $g_{A_7}(n)$, see Figure 2 and Figure 3, respectively.
    \begin{center}
\begin{figure}[!htb]
   \begin{minipage}{0.5\textwidth}
     \centering
     \includegraphics[width=1\linewidth]{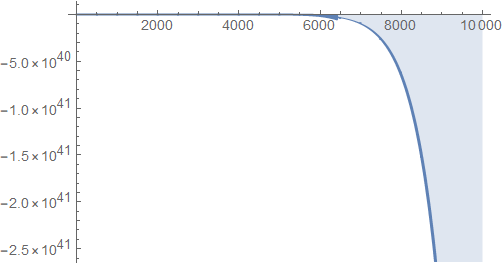}
     \caption{The values of $g_{A_6}(n)$ for $1\leqslant n \leqslant10^4$.}
   \end{minipage}\hfill
   \begin{minipage}{0.5\textwidth}
     \centering
     \includegraphics[width=1\linewidth]{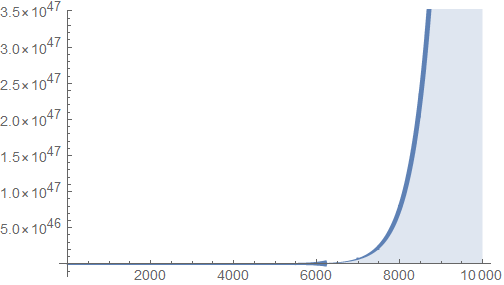}
     \caption{The values of $g_{A_7}(n)$ for $1\leqslant n \leqslant10^4$.}
   \end{minipage}
\end{figure}
\end{center}
\end{ex}

\section{The Laguerre inequalities for quasi-polynomial-like functions}

Now, we focus on the Laguerre inequalities for quasi-polynomial-like functions. As it was mentioned at the end of Section 2, we need to use a few binomial coefficient identities to deal with the issue. 

The first of them arises from comparing the coefficients of the expansions of both $(1-z)^{s}(1+z)^{s}$ and $(1-z^2)^{s}$ (see, \cite[Section 5.4]{GKP}).

\begin{lm}\label{Lemma 1}
    Let $s\in\mathbb{N}$ be fixed. Then for every even integer $0\leqslant n\leqslant s$, we have
    \begin{align*}
        \sum_{j=0}^n(-1)^j\binom{s}{j}\binom{s}{n-j}=(-1)^\frac{n}{2}\binom{s}{n/2}.
    \end{align*}
\end{lm}

To present the second one, we need to recall that the Stirling number of the second kind $\stirlingtwo{n}{k}$ enumerates the number of ways to partition a set of $n$ labelled objects into $k$ non-empty unlabelled subsets. Equivalently, it is the number of different equivalence relations with exactly $k$ equivalence classes that may be defined on an $n$ element set. It is worth noting that the following identities
\begin{align*}
    \stirlingtwo{n}{1}=1, \hspace{0.5cm}\stirlingtwo{n}{n}=1, \hspace{0.5cm}\stirlingtwo{n}{0}=0 \hspace{0.3cm}\text{ and } \hspace{0.3cm} \stirlingtwo{0}{0}=1
\end{align*}
hold for every positive integer $n$ as well as $\stirlingtwo{m}{k}=0$ whenever $0\leqslant m <k$. The succeeding lemma, together with the general introduction to the Stirling numbers, might be found in \cite[Section 6.1]{GKP}. 

\begin{lm}\label{Lemma 2}
    Let $u$ and $v$ be arbitrary non-negative integers. Then, we have
    \begin{align*}
        u!\stirlingtwo{v}{u}=\sum_{k=0}^u(-1)^{u-k}\binom{u}{k}k^v.
    \end{align*}
\end{lm}

Now, we are ready to state and prove the main result of this section.

\begin{thm}\label{theorem: Laguerre quasi}
    Let $f(n)$ be a quasi-polynomial-like function of the form
    \begin{align*}
        f(n)=c_{l}n^{l}+c_{l-1}n^{l-1}+\cdots+c_{l-2d}n^{l-2d}+o(n^{l-2d}),
    \end{align*}
    for some non-negative integer $d$ such that $2d\leqslant l$. Then, for every $0 \leqslant j \leqslant d$ the sequence $(f(n))_{n=0}^\infty$ satisfies the Laguerre inequality of order $j$ for all but finitely many values of $n$. In particular, we have that
    \begin{align*}
        \sum_{i=0}^{2d}(-1)^{i+d}\binom{2d}{i}f(n+i)f(n+2d-i)=(2d)!\binom{l}{d}c_l^2n^{2(l-d)}+o\left(n^{2(l-d)}\right).
    \end{align*}
\end{thm}
\begin{proof}
    Let us fix a quasi-polynomial-like function $f(x)$ as in the statement, and expand the left hand side of the inequality (\ref{def: Discrete Laguerre}) with $\omega_n=f(n)$. We have that
    \begin{align*}
        &\sum_{j=0}^{2d}(-1)^{j+d}\binom{2d}{j}f(n+j)f(n+2d-j)\\
        =&\sum_{j=0}^{2d}(-1)^{j+d}\binom{2d}{j}\left[c_{l}(n+j)^{l}+\cdots+c_{l-2d}(n+j)^{l-2d}+o\left((n+j)^{l-2d}\right)\right]\\
        &\times\left[c_{l}(n+2d-j)^{l}+\cdots+c_{l-2d}(n+2d-j)^{l-2d}+o\left((n+2d-j)^{l-2d}\right)\right]\\
        =&\sum_{j=0}^{2d}(-1)^{j+d}\binom{2d}{j}\left[c_{l}\sum_{i=0}^{2d}\binom{l}{i}j^in^{l-i}+\cdots+c_{l-2d}n^{l-2d}+o(n^{l-2d})\right]\\
        &\phantom{\sum_{j=0}^{2d}}\times\left[c_{l}\sum_{i=0}^{2d}(-1)^i\binom{l}{i}j^i(n+2d)^{l-i}+\cdots+c_{l-2d}(n+2d)^{l-2d}+o(n^{l-2d})\right].
    \end{align*}
    Since we are interested in the asymptotic behavior of the above expression, we need to determine the leading coefficient of its polynomial part. It is not difficult to notice that whenever we multiply a summand $\gamma_{i_0,k_0}j^{k_0}n^{l-i_0}$ from the first square bracket with a summand  $\delta_{i_1,k_1}j^{k_1}(n+2d)^{l-i_1}$ from the second one (where the coefficients $\gamma_{i_0,k_0}$ and $\delta_{i_1,k_1}$ are independent of both $j$ and $n$), we can obtain at most $j$ to the power $i_0+i_1\geqslant k_0+k_1$. More precisely, we get an expression of the form
    \begin{align}\label{proof: identity 1}
        \sum_{j=0}^{2d}(-1)^{j+d}\binom{2d}{j}\gamma_{i_0,k_0}\delta_{i_1,k_1}n^{l-i_0}(n+2d)^{l-i_1}j^{k_0+k_1},
    \end{align}
    where $0\leqslant k_0+k_1\leqslant i_0+i_1$. Therefore if $i_0+i_1<2d$, then (\ref{proof: identity 1}) might be rewritten as
    \begin{align*}
    (-1)^d\gamma_{i_0,k_0}\delta_{i_1,k_1}n^{l-i_0}(n+2d)^{l-i_1}\sum_{j=0}^{2d}(-1)^{2d-j}\binom{2d}{j}j^{k_0+k_1}=0,
    \end{align*}
    where the equality follows from Lemma \ref{Lemma 2} with $k=j$, $u=2d$ and $v=k_0+k_1$. Hence, our task boils down to finding the coefficient of $n^{2(l-d)}$. Repeating the above discussion, one can observe that the only possible non-zero term takes the form
    \begin{align*}
        &\sum_{j=0}^{2d}(-1)^{j+d}\binom{2d}{j}c_l^2\sum_{i=0}^{2d}(-1)^i\binom{l}{i}\binom{l}{2d-i}j^{2d}n^{2(l-d)}\\
        =(-1)^dc_l^2\times&\sum_{i=0}^{2d}(-1)^i\binom{l}{i}\binom{l}{2d-i}\times\sum_{j=0}^{2d}(-1)^{2d-j}\binom{2d}{j}j^{2d}.
    \end{align*}
    Now, Lemma \ref{Lemma 1} asserts that 
    \begin{align*}
        \sum_{i=0}^{2d}(-1)^i\binom{l}{i}\binom{l}{2d-i}=(-1)^d\binom{l}{d}.
    \end{align*}
    On the other hand, Lemma \ref{Lemma 2} maintains that
    \begin{align*}
        \sum_{j=0}^{2d}(-1)^{2d-j}\binom{2d}{j}j^{2d}=(2d)!\stirlingtwo{2d}{2d}=(2d)!.
    \end{align*}
    In conclusion, we obtain that
    \begin{align*}
        \sum_{j=0}^{2d}(-1)^{j+d}\binom{2d}{j}f(n+j)f(n+2d-j)=(2d)!\binom{l}{d}c_l^2n^{2(l-d)}+o\left(n^{2(l-d)}\right),
    \end{align*}
    which was to be demonstrated.
\end{proof}

As an immediate consequence of Theorem \ref{theorem: Laguerre quasi}, we get an analogue characterization to that one from Theorem \ref{theorem: higher order Turan A-partition function}.

\begin{thm}\label{theorem: Laguerre A-partition function}
    Let $A$ be a finite multiset (or set) of positive integers with $\#A=k$, and let $1\leqslant 2d<k$ be fixed. Suppose further that $\gcd B=1$ for every $(k-2d)$-multisubset $B\subset A$. Then, for each $1\leqslant j \leqslant d$ the sequence $(p_A(n))_{n=0}^\infty$ satisfies the Laguerre inequality of order $j$ for all but finitely many values of $n$.
\end{thm}
\begin{proof}
    The criterion easily follows from both Theorem \ref{2.4} and Theorem \ref{theorem: Laguerre quasi}.
\end{proof}

Analogously to Section 3, we present a few examples showing that Theorem \ref{theorem: Laguerre quasi}, as well as Theorem \ref{theorem: Laguerre A-partition function}, does not deliver us a necessary condition for the Laguerre inequality of order $d$ for $d\geqslant2$.

\begin{ex}
    Let us assume that $f(n)$ is a quasi-polynomial-like function of the form
    \begin{align*}
        f(n)=n^{10}+n^9+n^8+n^7+\left(n\pmod*{5}\right)\cdot n^6+o(n^6).
    \end{align*}
    It is not difficult to see that the assumption from Theorem \ref{theorem: Laguerre quasi} for $d=2$ does not hold. Nevertheless, one can calculate that
    \begin{align*}
        3 f(n + 2)^2 - 4f(n + 1)f(n + 3) + f(n)f(n + 4)\geqslant 525n^{16}+o(n^{16})
    \end{align*}
    for all sufficiently large values of $n$, and observe that the second Laguerre inequality is asymptotically satisfied for $f(n)$. Thus, Theorem \ref{theorem: Laguerre quasi} is not an optimal criterion.
\end{ex}

As a counterexample for Theorem \ref{theorem: Laguerre A-partition function}, we exhibit the following.

\begin{ex}\label{Example: 1,1,1,1,300(2)}
    We put $A=\{1,\text{$\textcolor{blue}{1},\textcolor{red}{1},\textcolor{brown}{1}$},300\}$ and consider the order $2$ Laguerre inequality. It is clear that the assumptions from Theorem \ref{theorem: Laguerre A-partition function} are not satisfied for $d=2$. Notwithstanding, if we set
    \begin{align*}
        h_{A}(n):=3p_{A}^2(n+2)-4p_{A}(n+1)p_{A}(n+3)+p_{A}(n)p_{A}(n+4),
    \end{align*}
    then it turns out that
    \begin{align*}
        h_A(n)\geqslant\frac{n^4}{2^3\cdot3^3\cdot5^4}+o(n^4)
    \end{align*}
    with the equality whenever $n\not\equiv297\pmod*{300}$, which agrees with Figure 4.
    \begin{center}
\begin{figure}[!htb]
   \begin{minipage}{1\textwidth}
     \centering
     \includegraphics[width=1\linewidth]{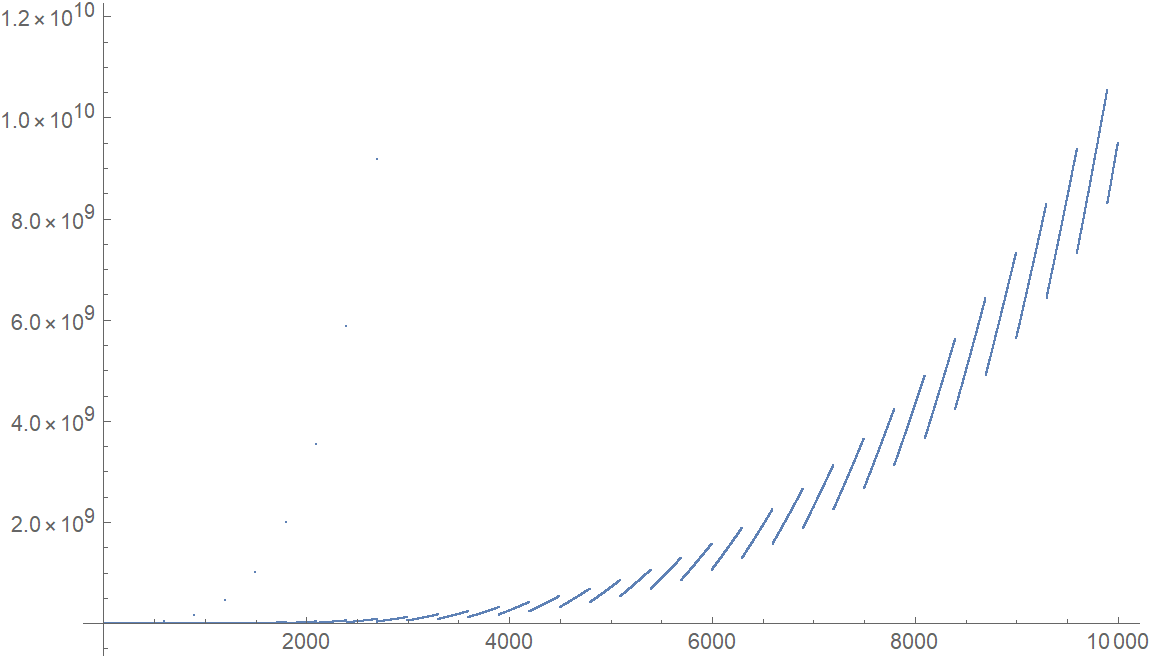}
     \caption{The values of $h_A(n)$ for $A=\{1,\text{$\textcolor{blue}{1},\textcolor{red}{1},\textcolor{brown}{1}$},300\}$ and $1\leqslant n \leqslant10^4$.}
   \end{minipage}\hfill
\end{figure}
\end{center}

In conclusion, we see that Theorem \ref{theorem: Laguerre A-partition function} is not an optimal criterion, as well.
\end{ex}

At the end of this section, we present an example showing that, in general, it might be difficult to derive an optimal criterion for the Laguerre inequality of order $d\geqslant2$ for quasi-polynomial-like functions. 

\begin{ex}\label{Example: 1,2,3,...,8,9}
    Let us consider the $A_m$-partition function defined in Example \ref{Example: 1,2,3,...}. For instance, we may examine the Laguerre inequality of order $2$ for both $p_{A_8}(n)$ and $p_{A_9}(n)$. For the sake of clarity, let us put
    \begin{align*}
        h_{A_m}(n):=3p_{A_m}^2(n+2)-4p_{A_m}(n+1)p_{A_m}(n+3)+p_{A_m}(n)p_{A_m}(n+4).
    \end{align*}
    In other words, $h_{A_m}(n)$ corresponds to the second order Laguerre inequality for $p_{A_m}(n)$. We see that the assumptions from Theorem \ref{theorem: Laguerre A-partition function} do not hold for $p_{A_8}(n)$ and $d=2$. Moreover, one can determine that
    \begin{align*}
        h_{A_8}(n)=-\frac{349n^{10}}{2^{20}\cdot3^{6}\cdot5^4\cdot7^3}+o\left(n^{10}\right),
    \end{align*}
    whenever $n\equiv0,2,\ldots,838\pmod*{840}$. In the case of $h_{A_9}(n)$, on the other hand, we get that the equality
    \begin{align*}
        h_{A_9}(n)=\frac{n^{12}}{2^{24}\cdot3^{11}\cdot5^4\cdot7^3}+o\left(n^{12}\right)
    \end{align*}
    holds for each positive integer $n$, which agrees with Theorem \ref{theorem: Laguerre A-partition function}. Figure 5 and Figure 6 exhibit the plots of $h_{A_8}(n)$ and $h_{A_9}(n)$, respectively.
    \begin{center}
\begin{figure}[!htb]
   \begin{minipage}{0.5\textwidth}
     \centering
     \includegraphics[width=1\linewidth]{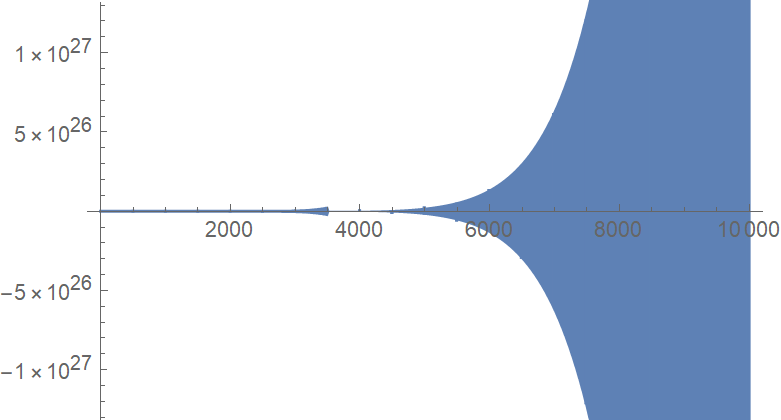}
     \caption{The values of $h_{A_8}(n)$ for $1\leqslant n \leqslant10^4$.}
   \end{minipage}\hfill
   \begin{minipage}{0.5\textwidth}
     \centering
     \includegraphics[width=1\linewidth]{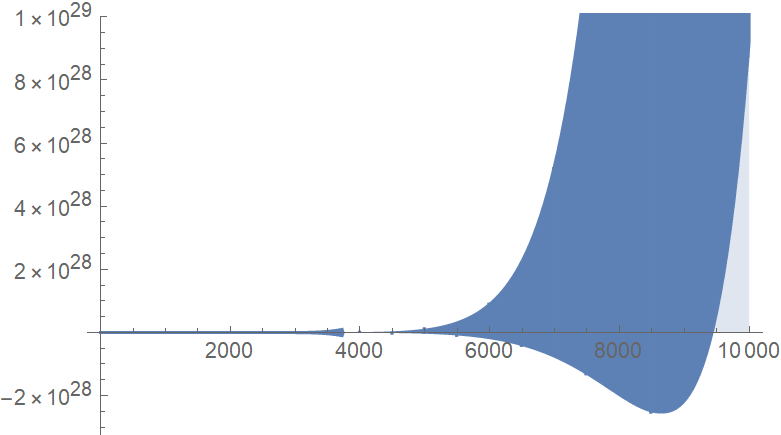}
     \caption{The values of $h_{A_9}(n)$ for $1\leqslant n \leqslant10^4$.}
   \end{minipage}
\end{figure}
\end{center}
The above discussion asserts that it might be difficult to find out an easy description of all quasi-polynomial-like functions which (asymptotically) fulfill the Laguerre inequality of order $d$ for any $d\geqslant2$.
\end{ex}

\section{Concluding remarks}

It is quite unfortunate that neither Theorem \ref{theorem: Criterion for quasi-polynomial-like functions} nor Theorem \ref{theorem: higher order Turan A-partition function} delivers necessary conditions for the order $d$ Tur\'an inequality. Analogously, neither Theorem \ref{theorem: Laguerre quasi} nor Theorem \ref{theorem: Laguerre A-partition function} contains necessary conditions for the Laguerre inequality of order $d$. It is worth pointing out that we have such a result in the case of the $r$-log-concavity problem for quasi-polynomial-like functions (and, in particular, $A$-partition functions) \cite{KG3}. Recall that a sequence of real numbers $\omega=\left(w_i\right)_{i=0}^\infty$ is called (asymptotically) $r$-log-concave for $r\in\mathbb{N}_+$, if there exists an integer $N$ such that all terms of the sequences
\begin{align*}
    \widehat{\mathcal{L}}\omega,\widehat{\mathcal{L}}^2\omega,\ldots,\widehat{\mathcal{L}}^r\omega
\end{align*}
are positive for every $i\geqslant N$, where
\begin{align*}
    \widehat{\mathcal{L}}\omega=\left(w_{i+1}^2-w_{i}w_{i+2}\right)_{i=0}^\infty \text{ and } \widehat{\mathcal{L}}^k\omega=\widehat{\mathcal{L}}\left(\widehat{\mathcal{L}}^{k-1}\omega\right)
\end{align*}
for $k\in\{2,3,\ldots,r\}$. We have the following characterization for that issue.

\begin{thm}[Gajdzica]
    Let $l$ and $r$ be arbitrary positive integers such that $l\geqslant2r$. Suppose further that we have 
    \begin{align*}
        f(n)=a_l(n)n^l+a_{l-1}(n)n^{l-1}+\cdots+a_{l-2r}(n)n^{l-2r}+o\left(n^{l-2r}\right),
    \end{align*} 
    where the coefficients $a_{l-2r}(n),\ldots,a_l(n)$ might depend on the residue class of\linebreak $n\pmod*{M}$ for some positive integer $M\geqslant2$. Then the sequence $\left(f(n)\right)_{n=0}^\infty$ is asymptotically $r$-log-concave if and only if all the numbers $a_{l-2r}(n),\ldots,a_l(n)$ are independent of the residue class of $n\pmod*{M}$.
\end{thm}

Unfortunately, the analogous descriptions are impossible for both the higher order Tur\'an inequalities and the Laguerre inequalities. For the former, if we assume that
\begin{align*}
        f(n)=c_{l}n^{l}+c_{l-1}n^{l-1}+\cdots+c_{l-d+1}n^{l-d+1}+o(n^{l-d+1})
    \end{align*}
    for some $1\leqslant d\leqslant l$. Then, the leading coefficient of the (quasi-polynomial-like) function which corresponds to the $d$th Tur\'an inequality may heavily depend on the residue class of $n$ modulo some positive integer, see Example \ref{Example: 1,1,1,1,300}. To visualize the issue more accurately, let us consider the third order Tur\'an inequality for  
    \begin{align*}
        f(n)=c_{l}n^{l}+c_{l-1}n^{l-1}+c_{l-2}n^{l-2}+c_{l-3}(n)n^{l-3}+o(n^{l-3}),
    \end{align*}
where $l\geqslant3$ and $c_{l-3}(n)$ depends on the residue class of $n\pmod*{M}$ for some $M\geqslant2$. It is tiresome but elementary to show that we have the following:
\footnotesize{
\begin{align*}
    &4\left(f(n)^2-f(n-1)f(n+1)\right)\left(f(n+1)^2-f(n)f(n+2)\right)-\left(f(n)f(n+1)-f(n-1)f(n+2)\right)^2\\
    &=\big{[}-6  c_{l-3}(n-1) c_{l-3}(n+1)+ 2 c_{l-3}(n-1) c_{l-3}(n+2)+4 l c_lc_{l-3}(n-1)- c_{l-3}^2(n-1)\\
    &\phantom{=\big{[}}+6 c_{l-3}(n-1)c_{l-3}(n) +6 c_{l-3}(n+1) c_{l-3}(n+2)+12 l c_l c_{l-3}(n+1)-9 c_{l-3}^2(n+1)\\
    &\phantom{=\big{[}}+18 c_{l-3}(n) c_{l-3}(n+1)-4 l c_l c_{l-3}(n+2)- c_{l-3}^2(n+2)-6 c_{l-3}(n) c_{l-3}(n+2)\\
    &\phantom{=\big{[}}+4 l^3 c_l^2-4 l^2 c_l^2-12 lc_l c_{l-3}(n) -9 c_{l-3}^2(n)\big{]}c_l^2n^{4l-6}+o(n^{4l-6}),
\end{align*}}\normalsize Thus, it is easy to see that the leading coefficient of that expression intensely depends on the residue class of $n\pmod*{M}$, which coincides with our discussion above. 

One can demonstrate a parallel reasoning to deduce that the same problem plagues us if we deal with the Laguerre inequality of order $d$ for $d\geqslant2$, which Example \ref{Example: 1,1,1,1,300(2)} indicates.  

At the end of the manuscript, we state a few open problems. The first of them encourages us to deal with the higher order Tur\'an inequalities for some particular $A$-partition functions.   

\begin{pro}\label{Problem 1}
    Fix a set (or multiset) $A$ of positive integers and investigate the higher order Tur\'an inequalities for the $A$-partition function.
\end{pro}

\begin{re}{\rm
    For instance, if we set $A=A_m=\{1,2,\ldots,m\}$ in Problem \ref{Problem 1}, then it is known \cite{KG2} that the second order Tur\'an inequality for $p_{A_m}(n)$ begins to hold for $m=5$. In that case we have that 
    \begin{align*}
        p_{A_5}^2(n)\geqslant p_{A_5}(n-1)p_{A_5}(n+1)
    \end{align*}
    for all $n>37$.

    One can extend the above and deal with the problem for the $A_m^{(l)}$-partition function, where $A_m^{(l)}=\{1^l,2^l,\ldots,m^l\}$ and $m\in \mathbb{N}_+\cup\{\infty\}$. The case of $m=\infty$ and $l=1$ has been investigated by Griffin et al. \cite{Griffin} and Larson and Wagner \cite{Larson}. For more information about the general setting (when $m=\infty$), we refer the reader to Ulas' paper \cite{MU2}. 
}
\end{re}

We also hope that there is a chance to discover a more efficient criterion than Theorem \ref{theorem: Criterion for quasi-polynomial-like functions}, and state the following.

\begin{pro}\label{Problem 2}
    Let $d\geqslant3$ be arbitrary. Find a more effective criterion than Theorem \ref{theorem: Criterion for quasi-polynomial-like functions} (or Theorem \ref{theorem: higher order Turan A-partition function}) for the order $d$ Tur\'an inequality for quasi-polynomial-like functions. Alternatively, do that for small values of the parameter $d$.
\end{pro}

Finally, we formulate the analogues of both Problem \ref{Problem 1} and Problem \ref{Problem 2} in the context of Laguerre inequalities for quasi-polynomial-like functions.

\begin{pro}
    Fix a set (or multiset) $A$ of positive integers and investigate the Laguerre inequalities for the $A$-partition function.
\end{pro}

\begin{pro}
    Let $d\geqslant2$ be arbitrary. Derive a more efficient criterion than Theorem \ref{theorem: Laguerre quasi} (or Theorem \ref{theorem: Laguerre A-partition function}) for the $d$th Laguerre inequality for quasi-polynomial-like functions. Alternatively, do that for small values of the parameter $d$.
\end{pro}

\section*{Acknowledgements}
I wish to express my sincere thanks to Piotr Miska and Maciej Ulas for their time and helpful suggestions. I am also grateful to Ian Wagner for his additional comments. This research was funded by both a grant of the National Science Centre (NCN), Poland, no. UMO-2019/34/E/ST1/00094 and a grant from the Faculty of Mathematics and Computer Science under the Strategic Program Excellence Initiative at the Jagiellonian University in Kraków.

\end{document}